\documentclass[a4paper,12p]{article}
\usepackage[english]{babel}
\usepackage[utf8]{inputenc}
\usepackage{enumerate}
\usepackage{amsmath}
\usepackage{amsthm}
\usepackage{hyperref}
\usepackage[capitalise]{cleveref}
\usepackage{eucal}
\usepackage{amssymb}
\usepackage{mathrsfs}
\usepackage{graphicx}
\usepackage{epstopdf}
\usepackage{float}
\usepackage{tikz}

\usetikzlibrary{arrows,chains,matrix,positioning,scopes,patterns}
\tikzset{node distance=3cm, auto}
\makeatletter
\tikzset{join/.code=\tikzset{after node path={%
\ifx\tikzchainprevious\pgfutil@empty\else(\tikzchainprevious)%
edge[every join]#1(\tikzchaincurrent)\fi}}}
\makeatother
\tikzset{>=stealth',every on chain/.append style={join},
         every join/.style={->}}

\usepackage{color}
\newtheorem{proposition}{Proposition}[section]
\newtheorem{theorem}{Theorem}[section]
\newtheorem{lemma}{Lemma}[section]

\theoremstyle{definition}
\newtheorem{definition}{Definition}[section]

\numberwithin{equation}{subsection}

\renewcommand{\hom}{\text{Hom}}

\author{Oscar Kivinen}
\title{Computations for Binomial Edge Ideals and Koszul Duality}
\date{\today}

\begin{document}
\maketitle

\section{Introduction}
A graded $k$-algebra $A$ is called Koszul if the residue field $A_0=k$ 
admits a linear free resolution as an $A$-module. Koszul algebras have 
gained plenty of interest, as they arise in several different 
mathematical contexts, such as combinatorics, algebraic geometry and 
representation theory, see for example \cite{froe} for a survey. 
Partly due to this, a remarkable fact about Koszul algebras is that they can be characterized in multiple ways, and that the class of Koszul algebras is closed under various natural operations.

Binomial edge ideals of (undirected, simple) graphs were introduced 
and studied by Herzog, Hibi et al. in the context of algebraic 
statistics in \cite{ci}. These provide a nice combinatorially 
characterized class of commutative quadratic algebras, and their 
algebraic properties have been discussed for example in 
\cite{cr,ecm,ek,sm}. The natural question of Koszulness of a binomial 
edge ideal $J_G$ can be transferred to combinatorial properties of the graph $G$, 
but lacks a full answer. Ene, Herzog and Hibi \cite{ek,ekf} investigated also some stronger notions of Koszulness for binomial edge ideals. Partial results in \cite{ek} give the 
implications \begin{center}
$G$ is closed $\Rightarrow$ $G$ is Koszul $\Rightarrow$ $G$ is chordal 
and claw-free.
\end{center}
These implications are strict, as a few counterexamples show. See 
\cref{gluingsect} for the precise definitions of these concepts. In this paper, we make some observations on binomial edge ideals, with the 
characterization of Koszulness as motivation.  In \cref{gluingsect}, we discuss building Koszul graphs from smaller pieces in a controlled manner. In \cref{cone}, the Koszul 
property of cone graphs is completely characterized. In \cref{resolution}, we compute the first two syzygies of the infinite minimal free resolution of $k$ over the algebra defined by a binomial edge ideal, and these happen to be always linear. In \cref{dual} we give a description of generators and relations for the Koszul or quadratic dual of the algebra defined by an arbitrary binomial edge ideal. 

\section{Binomial edge ideals of closed and Koszul graphs}\label{gluingsect}

For the rest of the article, we fix a simple graph $G=(V,E)$ on $n$ vertices, ie. one without loops or multiple edges, and label the vertices with $[n]$. To each vertex $i$, associate the variables $x_i,y_i$. The binomial edge ideal associated to $G$ is defined as $J_G:=\langle f_{ij}:=x_iy_j-y_jx_i |\{i,j\} \in E(G)\rangle \subset S:=k[x_1,\ldots,x_n,y_1,\ldots,y_n]$. For simplicity, we assume $k$ is a field of characteristic zero. 
 
First off, we collect some previous results on the algebraic and combinatorial properties of binomial edge ideals. Herzog, Hibi et al. gave in \cite{ci} the following complete characterization of the Gröbner bases of binomial edge ideals with respect to the lexicographic order induced by the labeling.
 
\begin{definition}
A path $i=i_0,\ldots,i_r=j$ with $i<j$ in a simple graph on $[n]$ is called \textit{admissible}, if \begin{enumerate}[(i)]
\item $i_l \neq i_k$ for $k \neq l$
\item for each $k=1,\ldots,r-1$ one has either $i_k<i$ or $i_k>j$
\item for any proper subset $\{j_1,\ldots,j_s\}$ of $\{i_1,\ldots,i_{r-1}\}$, the sequence $i,j_1,\ldots,j_s,j$ is not a path.
\end{enumerate}
\end{definition}
To each admissible path $\pi$ we associate a monomial 
$u_\pi=\prod_{i_k>j} x_{i_k}\prod_{i_l<i}y_{i_l}$. We then have the 
following theorem. \begin{theorem} $$\mathcal{G}=\bigcup_{i<j}\{u_\pi 
f_{ij}|\pi \text{ is an admissible path from i to j}\}$$ is a reduced 
Gröbner basis of $J_G$ for any simple $G$ on $[n]$, with respect to the lexicographic ordering induced by the labeling of the vertices. \end{theorem} The 
proof is somewhat long and can be found in \cite{ci}, Theorem 2.1.. Theorem 1.1. of the same paper shows that this Gröbner basis is quadratic if and only if $G$ is closed. Recall that a closed graph is one which admits a labeling on $[n]$ so that for all $\{i,j\}, \{k,l\} \in E$ with $i<j$ and $k<l$ one has $\{j,l\} \in E(G)$ if $i=k$, but $j=l$, and $\{i,k\} \in E(G)$ if $j=l$ but $i \neq k$. This can also be shown to be equivalent to the condition that for $1\leq i <j \leq n$ the shortest path between $i$ and $j$ is monotone.

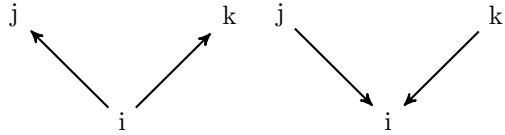
\begin{figure}[here]
\begin{center}
\begin{tikzpicture}[->,>=stealth',shorten >=1pt,auto,node distance=2cm,
  thick,main node/.style={draw=none,fill=none},scale=0.7]

  \node[main node] (1) {i};
  \node[main node] (2) [above left of=1] {j};
  \node[main node] (3) [above right of=1] {k};
  \node[main node] (4) [right of=1, node distance=3.5cm] {i};
  \node[main node] (5) [above left of=4] {j};
  \node[main node] (6) [above right of=4] {k};

  \path[every node/.style={font=\sffamily\small}]

    (1) edge node {} (2)
        edge node {} (3)
    (5) edge node {} (4)
    (6) edge node {} (4);
    
\end{tikzpicture}
\caption{Visualization of the characterization of closed graphs. In case there are directed edges coming from the labeling as shown, $\{j,k\}$ must also be an edge.}
\end{center}
\end{figure}

In \cite{cr}, Crupi and Rinaldo show that $G$ is closed if and only if the clique complex $\Delta(G)$ admits a labeling in which the facets are intervals. For example, note that the graph in \cref{internalgluing} admits only two such labelings. In \cite{ek}, Ene used 
these Gröbner bases and some arguments from \cite{con}, to show that gluing in the obvious sense along a free vertex preserves Koszulness. We call a vertex free if it is contained in a unique facet of $\Delta(G)$.
\begin{proposition}\label{gluing}
Let $G_1, G_2$ be two graphs with free vertices $v, v'$ 
respectively. Then the graph $G=(V,E)$ with vertex set 
$V(G_1)\cup (V(G_2)\backslash \{v'\})$ and edge set 
$$E(G_1)\cup (E(G_2)\backslash \{\{i,v'\} \in E(G_2)\}) 
\cup \{\{i,v\}|\{i,v'\} \in E(G_2)\}$$ formed by gluing 
$G_1, G_2$ along $v, v'$ is Koszul if and only if $G_1$ 
and $G_2$ are Koszul.\end{proposition} For a proof, see 
\cite{ek}, Theorem 2.4.. We will simply note that the 
proof relies on the fact that a quadratic $k$-algebra is 
Koszul if and only if for any quadratic or linear 
nonzerodivisor $f\in A$, the algebra $A/\langle f\rangle$ is Koszul.

A natural question that arises is: does this generalize to higher codimension minors, ie. does gluing along for example a free edge or free triangle preserve Koszulness? For plenty of small graphs that can be decomposed to Koszul or closed pieces, we have not yet been able to find a counterexample. With the aid of the computer algebra package Macaulay2 \cite{m2}, we computed examples up to graphs with $9$ vertices, and also got some hints on Koszulness by considering the beginning of the Koszul complex over $S/J_G$. However, this conjecture is not amenable to the same techniques as the proof of \cref{gluing}, since we cannot get a similar regular sequence anymore. We will next try to make this precise. 

In order to simplify things, one can consider coning free 
minors of $G$ to get new graphs that might be Koszul. Again, we are 
yet to find a counterexample. The naive approach to this, or gluing 
in higher codimension if we restrict our graphs enough, is to 
consider adding edges one-by-one, ie. to create new algebras 
$A/\langle f_e\rangle$ for some $e \notin E(G)$ and $A:=S/J_G$. This is due to the results in \cite{bf}, which state that if $f_e$ is strongly free or a zerodivisor and $A$ is Koszul, then $A/\langle f_e \rangle$ is Koszul. (An element $f_e\in A$ is called strongly free if $A$ has the same Hilbert series as the free product of $A/\langle f_e \rangle$ and $k[f_e]$.)

Both conditions can be seen 
to fail as follows: after relabeling, we may assume that our free 
minor is in the facet $[n-i,n]$, moreover on the vertices $[n-
j,n]$. Now coning a free vertex in a Koszul graph is the same as 
gluing an edge, call the resulting graph $G'$, with $V(G')=V(G)\cup 
\{n+1\}$. The addition of an internal edge, say $\{n+1,n\}$, 
corresponds algebraically to $S/J_{G'}\to S/J_G'+\langle 
f_{n+1,n}\rangle$. But the simple example in \cref{internalgluing} shows that in general, $f_e$ is not a nonzerodivisor nor strongly free, as can be computed with the computer algebra system Macaulay2 \cite{m2}.
\begin{figure}
\begin{center}
\begin{tikzpicture}[-,>=stealth',shorten >=1pt,auto,node distance=0.2cm,
  thick,main node/.style={circle,thick,draw=black,fill=black,inner sep=0pt,minimum width=3pt,minimum height=3pt},scale=.7]
  
\node[main node] (1) at (0,0) {};
\node[main node,fill=red, thin] (2) at (2,0) {};
\node[main node] (3) at (1,1) {};
\node[main node,fill=red, thin] (4) at (1,2) {};

\draw (3) -- (2) -- (1) -- (3) -- (4);
\draw[->] (3,0.5) -- (4,0.5);

\node[main node] (1a) at (5,0.5) {};
\node[main node] (2a) at (6,1.5) {};
\node[main node,fill=red,thin] (3a) at (6,-0.5) {};
\node[main node,fill=red,thin] (4a) at (7,0.5) {};

\draw (1a) -- (2a) -- (3a) -- (1a)
(2a)--(4a)--(3a);

\end{tikzpicture}
\caption{Example of a graph where adding an internal edge retains Koszulness, even closedness, but has a specific obstruction to an algebraic proof.\label[figure]{internalgluing}}
\end{center}
\end{figure}
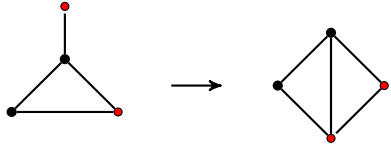

\section{Koszul cone graphs}

The main aim of this section is to classify all Koszul graphs that are cones. In \cite{con}, Rauf and Rinaldo considered the unmixed and Cohen-Macaulay properties for binomial edge ideals of cones. Cox and Erskine showed that a graph $G$ is closed if and only if it is chordal, claw--free and narrow, see \cite{cox}, Theorem 4.1.. That is, that $G$ has no induced cycles of length $>3$, no induced claws, and that the distance from each vertex to any longest shortest path is at most one.

We observe that for a simple graph $G$, the graph $cone(v,G)$ is claw-free if and only if the independence complex $Ind(G)$ is triangle-free. This is because if there is $\{a,b,c\} \in Ind(G)$, then $\{a,b,c,v\}$ is a claw with edges $\{av,bv,cv\}$ in $cone(v,G)$. Conversely, we may consider $G$ claw-free, and if $cone(v,G)$ has a claw, it is of the form $\{a,b,c,v\}$ with edges $\{av,bv,cv\}$, hence $\{a,b,c\}$ must be a triangle in $Ind(G)$.

Furthermore, we always have $diam(cone(v,G))\leq 2$. Especially, the distance of a vertex from a longest shortest path cannot be greater than one, so any cone is narrow.

\begin{proposition}\label{cone}
Let $G$ be simple. Then for $G'=cone(v,G)$ the following are equivalent:\begin{enumerate}
\item $G'$ is Koszul
\item $G'$ is closed
\item $Ind(G)$ is triangle-free and $G$ is chordal.
\end{enumerate}
\end{proposition}
\begin{proof}
Suppose $G'$ is Koszul. Then it is claw--free and chordal, by \cite{ek}, Theorem 1.1.. But by the second observation, $G'$ is then closed. The converse is obvious. If $G'$ is closed, it is chordal, claw-free and narrow. We see that $G'$ is chordal iff $G$ is, so by the first observation above $Ind(G)$ is triangle-free and $G$ is chordal. Finally, the converse follows from the second observation.
\end{proof}

\section{The quadratic dual and the Koszul resolution}
In this section, we describe the quadratic dual of binomial edge 
ideals for an arbitrary graph. In \cref{resolution}, we also 
explicitly compute the beginning of the  Koszul complex for 
$R=S/J_G$. Let $V\cong k^{2n}$ be a vector space, and $T(V)=V\oplus 
V^{\otimes 2} \oplus V^{\otimes 3}\oplus \ldots$ its tensor 
algebra. Choosing a basis, we can consider $T(V)=k\langle x_1,
\ldots,x_n,y_1,\ldots,y_n\rangle$, a noncommutative polynomial 
algebra. Letting $I$ be a quadratic (left) ideal $I \subset 
V\otimes V$ and $R=T(V)/I$ the algebra it defines, we dualize the 
exact sequence $$ 0 \to I \to V \otimes V \to (V \otimes V)/I \to 
0$$ with respect to $k$, to get the sequence $$0 \leftarrow 
(V^*\otimes V^*)/I^\perp \leftarrow V^*\otimes V^* \leftarrow 
I^\perp \leftarrow 0,$$ where the quadratic dual of $I$ is 
$$I^\perp:=\{\phi \in \hom(T(V),k)|\phi(I)=0\}.$$ Using a 
characterization in \cite{froe}, we have the following theorem.
\begin{theorem}\label{dual}
The quadratic dual of the algebra $S/J_G$ defined by the binomial edge ideal of $G$ is $$ k\langle x_1,\ldots,x_n,y_1,\ldots,y_n\rangle/J_G^\perp,$$ where $J_G^\perp$ is generated by the relations \begin{align*}
 x_i^2=0, y_i^2=0, \; 1\leq i \leq n,\\ x_ix_j+x_jx_i=0,y_iy_j+y_jy_i=0, \; 1\leq j<i\leq n,\\ x_iy_j+y_jx_i=0, \; \{i,j\} \notin E(G),\\ x_iy_j+y_jx_i+x_jy_i+y_ix_j=0,\; \{i,j\}\in E(G).
\end{align*}
\end{theorem}
\begin{proof}
Following \cite{froe}, Chapter 2, we have that for a commutative quadratic algebra $R\cong k[X_1,\ldots,X_m]/\langle f_1,\ldots,f_r\rangle$, where $f_i=\sum_{j\leq k}b_{ijk}X_jX_k$, the dual is given by $k\langle Y_1,\ldots,Y_m\rangle/J$, where $J=\langle g_1,\ldots, g_s\rangle$, $s=\binom{n}{2}-r$.  Here $g_i=\sum_{j\leq k} c_{ijk}[Y_i,Y_j]$, $[Y_i,Y_j]=Y_iY_j+Y_jY_i$ and the matrices $c_{ijk}$, $i=1,\ldots,s$ are a basis for the solutions of the linear system $\sum_{j\leq k} b_{ijk}Z_{jk}=0$, $i=1,\ldots, r$. Applying this to $S/J_G$, by relabeling $y_i=X_{n+i}$, we are reduced to solving the equation $Z_{j,n+k}=Z_{n+j,k}$, for $\{i,j\}\in E(G)$. The other $Z_{jk}$ give linearly independent solutions, hence the equations $[Y_i,Y_j]=0$ for $j\leq i$. Relabeling again, we arrive at the relations listed.
\end{proof}

\begin{lemma}\label{resolution}
Let $R=S/J_G$ for a binomial edge ideal. Then the start of the minimal free resolution of $k$ over $R$ is of the form $$\cdots \to R^{\beta_2} \to R^{\beta_1} \to 0,$$ where $\beta_1=2n$, $\beta_2=\binom{2n}{2}+|E|$.
\end{lemma}
\begin{proof}
The first rank is trivial. For the second one, we need 
to solve the linear equation $a \cdot \chi=a_1\chi_1+\ldots + 
a_{2n}\chi_{2n}=0$ over $R$, where $\chi_i=x_i, 
i=1, \ldots, n$ and $\chi_i=y_{i-n}, i=n+1,\ldots, 2n$ are the 
basis elements of $R^{\beta_1}$. As $J_G$ is a quadratic ideal, 
note that any $a_i$ must in the minimal case have degree $1$. Then $a=(0,
\ldots,-\chi_j, 0, \ldots, 0, \chi_i, \ldots, 0)$ where the 
$\chi_i, \chi_j$ are in $j$th and $i$th positions, respectively, give the $\binom{2n}{2}$ solutions. Inspecting quadratic zerodivisors in our ring furthermore shows that the syzygies given by $(0,\ldots,y_j,\ldots,-y_i,\ldots,0)$ also give basis elements for solutions, and indeed we cannot get more than $|E|$ of these.
\end{proof}

\end{document}